\documentclass{amsart}
\usepackage{graphicx}
\usepackage{latexsym}
\usepackage{amsfonts}
\usepackage[all]{xy}
\usepackage{amssymb, mathrsfs, amsfonts, amsmath}
\usepackage{amsbsy}
\usepackage{amsfonts}
\newtheorem*{acknowledgement}{Acknowledgement}
\newtheorem{corollary}{Corollary}

\newtheorem{definition}{Definition}
\newtheorem{lemma}{Lemma}
\newtheorem{proposition}{Proposition}

\newtheorem{theorem}{Theorem}

\numberwithin{equation}{section}

\begin{document}
\title[Miao-Tam critical metric]{Bach-flat critical metrics of the volume functional on 4-dimensional manifolds with boundary}
\author{A. Barros$^{1}$, R. Di\'ogenes$^{2}$ \& E. Ribeiro Jr.$^{3}$}
\address{$^{1,2}$ Universidade Federal do Cear\'a - UFC, Departamento  de Matem\'atica, Campus do Pici, Av. Humberto Monte, Bloco 914,
60455-760-Fortaleza / CE , Brazil.} \email{abarros@mat.ufc.br} \email{rafaeljpdiogenes@gmail.com}
\thanks{$^{1}$ Partially supported by CNPq/Brazil}
\thanks{$^{2}$ Partially supported by FUNCAP/Brazil}
\address{$^{3}$ Current: Department of Mathematics, Lehigh University, Bethlehem / PA, 18015, United States\\ Permanent: Universidade Federal do Cear\'a - UFC, Departamento  de Matem\'atica, Campus do Pici, Av. Humberto Monte, Bloco 914,
60455-760-Fortaleza / CE , Brazil.}\email{ernani@mat.ufc.br}
\thanks{$^{3}$ Partially supported by grants from  FUNCAP/Brazil and CNPq/Brazil}
\date{November 8, 2013}

\begin{abstract}
The purpose of this article is to investigate Bach-flat critical metrics of the volume functional on a compact manifold $M$ with boundary $\partial M.$ Here, we prove that  a Bach-flat critical metric of the volume functional on  a simply connected 4-dimensional manifold with boundary isometric to a standard sphere must be isometric to a geodesic ball in a simply connected space form $\Bbb{R}^{4},$ $\Bbb{H}^{4}$ or $\Bbb{S}^{4}$. Moreover, we show that in dimension three the result even is true replacing the Bach-flat condition by the weaker assumption that $M$ has divergence-free Bach tensor.
\end{abstract}

\maketitle

\section{Introduction}

A fruitful problem in Riemannian geometry is to study the critical points of the volume functional associated to space of smooth Riemannian structures. In the last decades very much attention has been given to study the critical points of the volume functional. Here, we shall study the space of smooth Riemannian structures on compact manifolds with boundary that satisfies a critical point equation associated to a boundary value problem.

Recently, inspired in a result obtained in \cite{fan} as well as in the characterization of the critical points of the scalar curvature functional, Miao and Tam studied variational properties of the volume functional constrained to the space of metrics of constant scalar curvature on a given compact manifold with boundary.  For more details, we refer the reader to \cite{miaotam} and \cite{miaotamTAMS}. Afterward, in a celebrated article \cite{CEM} Corvino, Eichmair and Miao studied this problem in a general context.  In fact, they studied the modified problem of finding stationary points for the volume functional on the space of metrics whose scalar curvature is equal to a given constant. To do this, they localized a condition satisfied by such stationary points to smooth bounded domains.

We now recall the definition of critical metrics studied by Miao and Tam. Here, for simplicity, these metrics will be called Miao-Tam critical metrics.

\begin{definition}
\label{def1} A Miao-Tam critical metric is a 3-tuple $(M^n,\,g,\,f),$ where $(M^{n},\,g),$ is a compact Riemannian manifold of dimension at least three with a smooth boundary $\partial M$ and $f: M^{n}\to \Bbb{R}$ is a smooth function such that $f^{-1}(0)=\partial M$ satisfying the overdetermined-elliptic system
\begin{equation}
\label{eqMiaoTam} \mathfrak{L}_{g}^{*}(f)=g.
\end{equation} Here, $\mathfrak{L}_{g}^{*}$ is the formal $L^{2}$-adjoint of the linearization of the scalar curvature operator $\mathfrak{L}_{g}$. Such a function $f$ is called a potential function.
\end{definition}

We recall that $\mathfrak{L}_{g}^{*}(f)=-(\Delta f)g+Hess\,f-f Ric;$  see for instance \cite{besse}. Therefore, the fundamental equation of a Miao-Tam critical metric (\ref{eqMiaoTam}) can be written as 
\begin{eqnarray}
\label{eqMiaoTam2}
-(\Delta f)g+Hess\,f-f Ric=g.
\end{eqnarray}

We highlight that some explicit examples of Miao-Tam critical metrics are in the form of warped products. Those examples include the spatial Schwarzschild metrics and AdS-Schwarzschild metrics restricted to certain domains containing their horizon and bounded by two spherically symmetric spheres (cf. Corollaries 3.1 and 3.2 in \cite{miaotamTAMS}).

It is not to hard to show that critical metrics have constant scalar curvature \cite{miaotam}. In 2009, Miao and Tam were able to prove that these metrics arise as critical points of the volume functional on $M^n$ when restricted to the class of metrics $g$ with prescribed constant scalar curvature such that $g\mid_{T \partial M}=h$ for a prescribed Riemannian metric $h$  on the boundary.

Here, we call attention to the paragraph where Miao and Tam \cite{miaotamTAMS} wrote:

\begin{flushright}
\begin{minipage}[t]{4.37in}
 \emph{``we want to know if there exist non-constant sectional curvature critical metrics on a compact manifold whose boundary is isometric to standard sphere. If yes, what can we say about the structure of such metrics?"} 
 \end{minipage}
\end{flushright}

Indeed, they studied these critical metrics under Einstein and conformally flat assumptions. In particular, they proved that a connected, compact, Einstein manifold $(M^n ,\,g)$ with smooth boundary that satisfies  (\ref{eqMiaoTam2}) must be isometric to a geodesic ball in a simply connected space form $\Bbb{R}^{n},$ $\Bbb{H}^{n}$ or $\Bbb{S}^{n}.$ Moreover, based on the techniques developed in a work of Kobayashi and Obata \cite{obata}, Miao and Tam showed that the result even is true replacing the Einstein condition by the assumption that $(M^n ,\,g)$ is locally conformally flat with boundary isometric to a standard sphere.  More precisely, they proved the following result.

\begin{theorem}[Miao-Tam, \cite{miaotamTAMS}]
\label{theoremMT}
Let $(M^n,g,f)$ be a  locally conformally flat simply connected, compact Miao-Tam critical metric with boundary isometric to a standard sphere $\Bbb{S}^{n-1}.$ Then $(M^n ,\,g)$ is isometric to a geodesic ball in a simply connected space form $\Bbb{R}^{n},$ $\Bbb{H}^{n}$ or $\Bbb{S}^{n}$.
\end{theorem}

It should be emphasized that the hypothesis that the boundary of $M^n$ is isometric to a standard sphere $\Bbb{S}^{n-1}$ considered by Miao and Tam is not artificial. To clarify this, we consider that the boundary of $M^n$ is totally geodesic and is isometric to a standard sphere $\Bbb{S}^{n-1}.$  Under these conditions,  motivated by the positive mass theorem, Min-Oo conjectured that if $M^n$ has scalar curvature at least $n(n-1),$ then $M^n$ must be isometric to the hemisphere $\Bbb{S}^{n}_{+}$ with standard metric (cf. \cite{MinO}). However, an elegant article due to Brendle, Marques and Neves  shows counterexamples to Min-Oo's Conjecture in dimensions $n\geq 3.$ For more details see \cite{marques}. We also highlight that  $\Bbb{S}^{n}_{+}$ satisfies (\ref{eqMiaoTam2}) for a suitable potential function (cf. \cite{miaotam} p. 153).

We now recall that  the Bach tensor on a Riemannian manifold $(M^n,g)$, $n\geq 4,$ which was introduced in the early 1920s to study conformal relativity, see \cite{bach},  is defined in term of the components of the Weyl tensor $W_{ikjl}$ as follows
\begin{equation}
\label{bach} B_{ij}=\frac{1}{n-3}\nabla^{k}\nabla^{l}W_{ikjl}+\frac{1}{n-2}R_{kl}W_{i}\,^{k}\,_{j}\,^{l},
\end{equation}
while for $n=3$ it is given by
\begin{equation}
\label{bach3} B_{ij}=\nabla^kC_{kij}.
\end{equation}
We say that $(M^n,g)$ is Bach-flat when $B_{ij}=0.$ On 4-dimensional compact manifolds, Bach-flat metrics are precisely critical points of the conformally invariant functional $\mathcal{W}(g)$ defined on the space of smooth Riemannian structures as follows
\begin{equation*}
\mathcal{W}(g)=\int_{M}|W_{g}|^2 dM_{g},
\end{equation*}
where $W_{g}$ denotes the Weyl tensor of $g$. It is not difficult to check that locally conformally flat metrics as well as Einstein metrics are Bach-flat. Recently, Cao and Chen have studied Bach-flat gradient Ricci solitons, more precisely, they showed a stronger classification for gradient Ricci solitons under the Bach-flat assumption. For more details, we refer the reader to \cite{CaoChen} and \cite{CaoChenT}.

It is well-known that 4-dimensional compact Riemannian manifolds have special behavior; for more details see for instance \cite{AMS}, \cite{besse} and \cite{scorpan}. Here, we shall investigate Bach-flat  critical metrics of the volume functional on 4-dimensional manifolds with boundary.  More precisely, we replace the assumption of locally conformally flat in the Miao-Tam result by the Bach-flat condition, which is weaker that the former. We now state our first result.

\begin{theorem}\label{thm1}
Let $(M^4,g,f)$ be a simply connected, compact Miao-Tam critical metric with boundary isometric to a standard sphere $\Bbb{S}^{3}.$ Then $(M^4 ,\,g)$ is isometric to a geodesic ball in a simply connected space form $\Bbb{R}^{4},$ $\Bbb{H}^{4}$ or $\Bbb{S}^{4}$ provided $$\int_M f^2B(\nabla f,\nabla f)dM_{g}\geq 0,$$ where $B$ is the Bach tensor.
\end{theorem}

The proof of Theorem \ref{thm1} was inspired in the trend developed by Cao and Chen in \cite{CaoChen}.  In the sequel, as an immediate consequence of Theorem \ref{thm1} we deduce the following corollary.

\begin{corollary}
Let $(M^4,g,f)$ be a Bach-flat simply connected, compact Miao-Tam critical metric with boundary isometric to a standard sphere $\Bbb{S}^{3}.$ Then $(M^4 ,\,g)$ is isometric to a geodesic ball in a simply connected space form $\Bbb{R}^{4},$ $\Bbb{H}^{4}$ or $\Bbb{S}^{4}$.
\end{corollary}

Based in the previous result, it is natural to ask what occurs in lower dimension. To do so, inspired in the ideas developed in \cite{CCCMM} (see also \cite{CaoChenT} and \cite{CaoChen}) we shall prove a rigidity result for a 3-dimensional Miao-Tam critical metric  with divergence-free Bach tensor, i.e. $div B= 0,$ and boundary isometric to a standard sphere $\Bbb{S}^2.$ Clearly, the assumption of divergence-free Bach tensor is weaker than the Bach-flat condition considered in Theorem \ref{thm1}. More precisely, we have the following result.

\begin{theorem}\label{thm2}
Let $(M^3,\,g,\,f)$ be a simply connected, compact Miao-Tam critical metric with boundary isometric to a standard sphere $\Bbb{S}^2.$ If $div B(\nabla f)=0$ in $M,$ where $B$ is the Bach tensor, then $(M^3,\,g)$ is isometric to a geodesic ball in a simply connected space form $\Bbb{R}^{3},$ $\Bbb{H}^{3}$ or $\Bbb{S}^3.$
\end{theorem}

Finally, we get the following rigidity result.

\begin{corollary}
\label{cor2}
Let $(M^3,\,g,\,f)$ be a simply connected, compact Miao-Tam critical metric  with divergence-free Bach tensor and boundary isometric to a standard sphere $\Bbb{S}^2.$ Then $(M^3,\,g)$ is isometric to a geodesic ball in a simply connected space form $\Bbb{R}^{3},$ $\Bbb{H}^{3}$ or $\Bbb{S}^3.$
\end{corollary}

\section{Preliminaries and Key Lemmas}

In this section we shall present a couple of  lemmas that will be useful in the proof of our main results.  We begin recalling that  $$ \mathfrak{L}_{g}^{*}(f)=-(\Delta f)g+Hess\,f-f Ric.$$ So, as it was previously mentioned the fundamental equation of a Miao-Tam critical metric (\ref{eqMiaoTam}) becomes
\begin{equation}
\label{eqfund1}
-(\Delta f)g+Hess f-fRic=g.
\end{equation}
Tracing (\ref{eqfund1}) we arrive at
\begin{equation}
\label{eqtrace}
(n-1)\Delta f+Rf=-n.
\end{equation}
Moreover, by using (\ref{eqtrace}) it is not difficult to check that
\begin{equation}
\label{IdRicHess} f\mathring{Ric}=\mathring{Hess} f,
\end{equation} where $\mathring{T}$ stands for the traceless of $T.$

For simplicity, we now rewrite equation (\ref{eqfund1}) in the tensorial language as follows
\begin{equation}
\label{fundeqtens} -(\Delta f)g_{ij}+\nabla_{i}\nabla_{j}f -fR_{ij}=g_{ij}.
\end{equation}

Next, since a Miao-Tam critical metric has constant scalar curvature (cf. \cite{miaotam}), we use the last identity in order to obtain the following lemma.
\begin{lemma}
\label{lem1} Let $\big(M^n,\,g,\,f)$ be a Miao-Tam critical metric. Then
\begin{equation*}
f\big(\nabla_{i}R_{jk}-\nabla_{j}R_{ik}\big)=R_{ijks}\nabla^{s}f + \frac{R}{n-1}\big(\nabla_{i}f
g_{jk}-\nabla_{j}f g_{ik}\big)- \big(\nabla_{i}f R_{jk}-\nabla_{j}f R_{ik}\big).
\end{equation*}
\end{lemma}
\begin{proof} Computing $\nabla_{i}(fR_{jk})$ with the aid of (\ref{fundeqtens}) we infer
\begin{equation}
\label{eq1lem1}
(\nabla_{i}f)R_{jk}+f\nabla_{i}R_{jk}=\nabla_{i}\nabla_{j}\nabla_{k}f-(\nabla_{i}\Delta f)g_{jk}.
\end{equation}
Since $R$ is constant (\ref{eqtrace}) yields $\nabla_{i}\Delta f=-\frac{R}{n-1}\nabla_{i}f$. Hence we use this data in (\ref{eq1lem1}) to deduce
\begin{equation}
\label{eq2lem2}
f\nabla_{i}R_{jk}=-(\nabla_{i}f)R_{jk}+\nabla_{i}\nabla_{j}\nabla_{k}f+\frac{R}{n-1}\nabla_{i}fg_{jk}.
\end{equation}
Therefore, it suffices to apply the Ricci identity to finish the proof of the lemma.
\end{proof}

To fix notations we recall three special tensors in the study of curvature for a Riemannian manifold $(M^n,\,g),\,n\ge 3.$  The first one is the Weyl tensor $W$ which is defined by the following decomposition formula
\begin{eqnarray}
\label{weyl}
R_{ijkl}&=&W_{ijkl}+\frac{1}{n-2}\big(R_{ik}g_{jl}+R_{jl}g_{ik}-R_{il}g_{jk}-R_{jk}g_{il}\big) \nonumber\\
 &&-\frac{R}{(n-1)(n-2)}\big(g_{jl}g_{ik}-g_{il}g_{jk}\big),
\end{eqnarray}
where $R_{ijkl}$ stands for the Riemann curvature operator, whereas the second one is the Cotton tensor $C$ given by
\begin{equation}
\label{cotton} \displaystyle{C_{ijk}=\nabla_{i}R_{jk}-\nabla_{j}R_{ik}-\frac{1}{2(n-1)}\big(\nabla_{i}R
g_{jk}-\nabla_{j}R g_{ik}).}
\end{equation}
These two tensors are related as follows
\begin{equation}
\label{cottonwyel} \displaystyle{C_{ijk}=-\frac{(n-2)}{(n-3)}\nabla_{l}W_{ijkl},}
\end{equation}provided $n\ge 4.$
Finally, the Schouten tensor  $A$ is defined by
\begin{equation}
\label{schouten} A_{ij}=\frac{1}{n-2}\left(R_{ij}-\frac{R}{2(n-1)}g_{ij}\right).
\end{equation}

Combining equations (\ref{weyl}) and (\ref{schouten}) we have the following splitting
\begin{equation}
\label{WS} R_{ijkl}=\frac{1}{n-2}(A\odot g)_{ijkl}+W_{ijkl},
\end{equation}
where $\odot$ is the Kulkarni-Nomizu product.  For more details about these tensors, we refer to \cite{besse}.

From now on we introduce the covariant 3-tensor $T_{ijk}$ by

\begin{eqnarray}
\label{T}
T_{ijk}&=&\frac{n-1}{n-2}\left(R_{ik}\nabla_{j}f-R_{jk}\nabla_{i}f\right)
-\frac{R}{n-2}\left(g_{ik}\nabla_{j}f-g_{jk}\nabla_{i}f\right)\nonumber\\
 &&+\frac{1}{n-2}\left(g_{ik}R_{js}\nabla^{s}f-g_{jk}R_{is}\nabla^{s}f\right).
\end{eqnarray}
It is important to highlight that $T_{ijk}$ was defined similarly to $D_{ijk}$ in \cite{CaoChen}.  Now, we may announce our second lemma.

\begin{lemma}\label{lemTCW} Let $(M^n,g,f)$ be a Miao-Tam critical metric. Then the following identity holds:
\begin{equation*}
fC_{ijk}=T_{ijk}+W_{ijks}\nabla^{s}f.
\end{equation*}
\end{lemma}
\begin{proof} First of all, we compare (\ref{cotton}) with Lemma \ref{lem1} to arrive at
\begin{equation}
\label{eq1pl2} fC_{ijk}=R_{ijks}\nabla^{s}f + \frac{R}{n-1}\big(\nabla_{i}f g_{jk}-\nabla_{j}f g_{ik}\big)-
\big(\nabla_{i}f R_{jk}-\nabla_{j}f R_{ik}\big).
\end{equation}
On the other hand, from  (\ref{weyl}) we obtain
\begin{eqnarray*}
R_{ijks}\nabla^{s}f&=&W_{ijks}\nabla^{s}f+\frac{1}{n-2}\big(R_{ik}g_{js}+R_{js}g_{ik}-R_{is}g_{jk}-R_{jk}g_{is}\big)\nabla^{s}f\\&&-\frac{R}{(n-1)(n-2)}\big(g_{js}g_{ik}-g_{is}g_{jk}\big)\nabla^{s}f.
\end{eqnarray*}
From this it follows that
\begin{eqnarray*}
fC_{ijk}&=&W_{ijks}\nabla^{s}f + \frac{(n-1)}{(n-2)}\big(R_{ik}\nabla_{j}f-R_{jk}\nabla_{i}f\big)-\frac{R}{(n-2)}\big(g_{ik}\nabla_{j}f - g_{jk}\nabla_{i}f\big)\\
&&+\frac{1}{n-2}\big(g_{ik}R_{js}\nabla^{s}f-g_{jk}R_{is}\nabla^{s}f\big)\\
 &=&T_{ijk}+W_{ijks}\nabla^{s}f,
\end{eqnarray*}which concludes the proof of the lemma.
\end{proof}

To simplify some computations we shall define a function $\rho$ on $M^n$ by
\begin{equation}\label{rho}
\rho=|\nabla f|^2+\frac{2}{n-1}f+\frac{R}{n-1}f^2.
\end{equation} We claim that 
\begin{equation}\label{gradrho}
\frac{1}{2}\nabla\rho=fRic(\nabla f).
\end{equation}Indeed, since $R$ is constant we have $
\frac{1}{2}\nabla\rho=Hess f(\nabla f) +\frac{1}{n-1}\nabla f +\frac{R}{n-1}f\nabla f.$ Next, we use (\ref{eqfund1}) and (\ref{eqtrace}) to obtain
$$
\frac{1}{2}\nabla\rho=Hess f(\nabla f) -(\Delta f+1)\nabla f=fRic(\nabla f),
$$
which settles our claim.

Proceeding we recall that, at regular points of a smooth function $f,$ the vector field $\nu=\frac{\nabla f}{\mid \nabla f \mid}$ is normal to $\Sigma_{c}=\{p\in M:f(p)=c\}.$ In particular the second fundamental form of $\Sigma_{c}$  is given by
\begin{equation}
\label{secff}
h_{ij}=-\langle \nabla_{e_{i}}\nu,e_{j}\rangle,
\end{equation}where $\{e_{1},\ldots,e_{n-1}\}$ is  an orthonormal frame on $\Sigma_{c}.$ Then the mean curvature computed at these points, denoted by  $H$, is given as follows

\begin{equation}
\label{meancurv1}
H=-\frac{1}{\mid \nabla f \mid}\sum_{i=1}^{n-1}Hessf (e_{i},e_{i}).
\end{equation}

We now follow  the trend of Cao and Chen  (cf. \cite{CaoChen} and  \cite{CaoChenT}) to study the level sets of the potential function of  Miao-Tam critical metrics. To this end, first, we  deduce a similar result concerning to the tensor $T$ defined by (\ref{T}) on the next lemma. 

\begin{lemma}\label{propT}
Let $(M^n,g,f)$ be a Miao-Tam critical metric. Let $\Sigma=\{f=f(p)\}$ be a level set of $f$. If $g_{ab}$ denotes  the induced metric on $\Sigma,$ then, at any point where $\nabla f\neq0,$ we have
\begin{equation*}
|fT|^2=\frac{2(n-1)^2}{(n-2)^2}|\nabla
f|^4\sum_{a,b=2}^{n}|h_{ab}-\frac{H}{n-1}g_{ab}|^2+\frac{n-1}{2(n-2)}|\nabla^{\Sigma}\rho|^2,
\end{equation*}
where $\rho$ is given by (\ref{rho}), $h_{ab}$ and $H$ are the second fundamental form and the mean curvature of $\Sigma,$ respectively, while $\nabla^{\Sigma}$ is the Riemannian connection of $\Sigma.$
\end{lemma}
\begin{proof}
We consider  an orthonormal frame $\{e_{1},e_2,\ldots,e_{n}\}$ with $e_1=\frac{\nabla f}{|\nabla f|}$ and $e_2,\ldots,e_n$ tangent to $\Sigma.$ A straightforward computation allows us to deduce
\begin{eqnarray*}
|T|^2&=&\frac{2(n-1)^2}{(n-2)^2}(|Ric|^2|\nabla f|^2-|Ric(\nabla f)|^2)+\frac{2(n-1)R^2}{(n-2)^2}|\nabla f|^2 \\&&+\frac{2(n-1)}{(n-2)^2}|Ric(\nabla f)|^2-\frac{4(n-1)R}{(n-2)^2}(R|\nabla f|^2-Ric(\nabla f,\nabla f))\\&&+\frac{4(n-1)}{(n-2)^2}(RRic(\nabla f,\nabla f)-|Ric(\nabla f)|^2)-\frac{4(n-1)R}{(n-2)^2}Ric(\nabla f,\nabla f).
\end{eqnarray*}
Proceeding we can use (\ref{gradrho}) to obtain
\begin{eqnarray}
\label{eqfT} |fT|^2&=&\frac{2(n-1)^2}{(n-2)^2}f^2 |\nabla f|^2 |Ric|^2 -\frac{n(n-1)}{2(n-2)^2}|\nabla\rho|^2\nonumber\\&
&-\frac{2(n-1)R^2}{(n-2)^2}f^2 |\nabla f|^2+\frac{2(n-1)R}{(n-2)^2}f\langle\nabla\rho,\nabla f\rangle.
\end{eqnarray}

On the other hand, the second fundamental form $h_{ab}$ of the level set $\Sigma$ as well as its mean curvature $H,$ are given, respectively, by
\begin{equation*}
h_{ab}=-\left\langle\nabla_{e_a}\left(\frac{\nabla f}{|\nabla f|}\right),e_b\right\rangle=-\frac{1}{|\nabla
f|}\nabla_a\nabla_b f=-\frac{1}{|\nabla f|}\left[fR_{ab}-\Big(\frac{1}{n-1}+\frac{fR}{n-1}\Big)g_{ab}\right]\end{equation*}
and
\begin{equation*}
H=-\frac{1}{|\nabla f|}(fR-fR_{11}-fR-1)=\frac{1}{|\nabla f|}(fR_{11}+1).
\end{equation*}
Whence, we deduce
\begin{eqnarray*}
|h|^2 &=&\frac{1}{|\nabla f|^2}\left[f^2|Ric|^2-2f^2\sum_{a=2}^{n}R_{1a}^{2}-f^2R_{11}^{2}-2f(R-R_{11})\Big(\frac{1}{n-1}+\frac{fR}{n-1}\Big)\right]\\&&+\frac{1}{|\nabla f|^2}\frac{(fR+1)^2}{n-1}
\end{eqnarray*}
and
\begin{equation*}
H^2=\frac{1}{|\nabla f|^2}\left(f^2 R_{11}^{2}+2fR_{11}+1\right).
\end{equation*}
After some computations we obtain
\begin{eqnarray}
\label{eq123}
\sum_{a,b=2}^{n}|h_{ab}-\frac{H}{n-1}g_{ab}|^2 &=&\frac{1}{|\nabla f|^2}\left[f^2|Ric|^2-\Big(\frac{nR_{11}^{2}f^2+f^2 R^2-2f^2 RR_{11}}{n-1}\Big)\right]\nonumber\\&&-\frac{2f^2}{|\nabla f|^2}\sum_{a=2}^{n}R_{1a}^{2}.
\end{eqnarray}

On the other hand, by using once more (\ref{gradrho}) we get
$$
fR_{11}=\frac{1}{|\nabla f|^2}fRic(\nabla f,\nabla f)=\frac{1}{2|\nabla f|^2}\langle\nabla\rho,\nabla
f\rangle$$
and
$$fR_{1a}=\frac{1}{|\nabla f|}fRic(\nabla f,e_{a})=\frac{1}{2|\nabla f|}\langle\nabla\rho,e_{a}\rangle=\frac{1}{2|\nabla f|}\nabla_{a}\rho.$$

Substituting the last two identities in (\ref{eq123}) we obtain
\begin{eqnarray*}
\sum_{a,b=2}^{n}|h_{ab}-\frac{H}{n-1}g_{ab}|^2&=&\frac{1}{|\nabla f|^2}\Big(f^2|Ric|^2-\frac{1}{2|\nabla
f|^2}|\nabla^{\Sigma}\rho|^2-\frac{R^2 f^2}{n-1}\\&-&\frac{n}{4(n-1)|\nabla f|^4}\langle\nabla\rho,\nabla f\rangle^2+\frac{R f}{(n-1)|\nabla f|^2}\langle\nabla\rho,\nabla f\rangle\Big),
\end{eqnarray*}
which can be rewritten as
\begin{eqnarray}
\label{eq235}
f^2|Ric|^2&=&|\nabla f|^2\sum_{a,b=2}^{n} |h_{ab}-\frac{H}{n-1}g_{ab}|^2+\frac{n}{4(n-1)|\nabla f|^4}\langle\nabla\rho,\nabla f\rangle^2\nonumber\\
 &&+\frac{1}{2|\nabla
f|^2}|\nabla^{\Sigma}\rho|^2+\frac{R^2 f^2}{n-1}-\frac{R f}{(n-1)|\nabla f|^2}\langle\nabla\rho,\nabla f\rangle.
\end{eqnarray}
Finally, comparing (\ref{eqfT}) with (\ref{eq235}), we deduce
\begin{equation*}
|fT|^2=\frac{2(n-1)^2}{(n-2)^2}|\nabla
 f|^4 \sum_{a,b=2}^{n}|h_{ab}-\frac{H}{n-1}g_{ab}|^2+\frac{n-1}{2(n-2)}|\nabla^{\Sigma}\rho|^2,
\end{equation*}
which completes the proof of the lemma.
\end{proof}

We point out that some of these calculations above were also done in \cite{br} and \cite{QingYuan} in a different context to study the CPE conjecture (cf. Besse \cite{besse}, page 128).

Next, as a consequence of  Lemma \ref{propT} we derive the following properties concerning a level set of the quoted metrics.

\begin{proposition}\label{propvanishT}
Let $(M^n,\,g,\,f)$ be a Miao-Tam critical metric with $T\equiv 0.$ Let $c$ be a regular value of $f$ and
$\Sigma=\{p\in M;\,f(p)=c\}$ be a level set of $f.$ We consider $e_1=\frac{\nabla f}{|\nabla f|}$ and choose
an orthonormal frame $\{e_2,\ldots,e_n\}$ tangent to  $\Sigma.$ Under these conditions the following assertions hold.
\begin{enumerate}
\item \label{item1pT} The second fundamental form $h_{ab}$ of $\Sigma$ is $h_{ab}=\frac{H}{n-1}g_{ab}.$
\item \label{item2pT} $|\nabla f|$ is constant on $\Sigma.$
\item \label{item3pT} $R_{1a}=0$ for any $a\geq2$ and $e_1$ is an eigenvector of $Ric.$
\item  \label{item4pT} The mean curvature of $\Sigma$ is constant. 
\item \label{item5pT} On $\Sigma,$ the Ricci tensor either has a unique eigenvalue or two distinct eigenvalues with
multiplicity $1$ and $n-1,$ respectively. Moreover, the eigenvalue with multiplicity $1$ is in the direction of $\nabla f.$
\item \label{item6pT} $R_{1abc}=0$, for $a,b,c\in\{2,\ldots,n\}.$
\end{enumerate}
\end{proposition}
\begin{proof}
The first two items follow directly from Lemma \ref{propT} jointly with (\ref{rho}). Since $T\equiv0$ we may use (\ref{T}) to deduce
\begin{eqnarray*}
0&=&T(e_i,\nabla f,\nabla f)\\
 &=&Ric(e_i,\nabla f)|\nabla f|^2-Ric(\nabla f,\nabla f)\langle \nabla f, e_{i}\rangle,
\end{eqnarray*}
in other words $$Ric(e_i,\nabla f)|\nabla f|^2=Ric(\nabla f,\nabla f)\langle \nabla f, e_{i}\rangle.$$ So, for $i=a\geq2$, we obtain $R_{1a}=0$. Furthermore $Ric(e_1)=g^{ij}R_{1j}e_i=R_{11}e_1.$ Therefore,  $e_1=\frac{\nabla f}{|\nabla f|}$ is an eigenvector of $Ric,$ which establishes the third assertion.

Proceeding we consider the Codazzi equation
\begin{equation}
\label{codazzi} R_{1abc}=\nabla^{\Sigma}_{b}h_{ca}-\nabla^{\Sigma}_{c}h_{ba}, \,\,\, a,b,c=2,\ldots,n.
\end{equation}
By contracting  (\ref{codazzi}) with respect to indices $a$ and $c,$ and using (\ref{item1pT}), we get
\begin{equation*}
R_{1b}=\nabla^{\Sigma}_{b}H-g^{ac}\nabla^{\Sigma}_{c}h_{ba}=\frac{n-2}{n-1}\nabla^{\Sigma}_{b}H.
\end{equation*}
Now, we use that $R_{1b}=0$ to conclude that $H$ is constant on $\Sigma,$ which gives the fourth item.  Next, since $e_1=\frac{\nabla f}{|\nabla f|}$ is an eigenvector of $Ric,$ we may choose the frame $$\{e_{1}=\frac{\nabla f}{|\nabla f|},e_{2},\ldots,e_{n}\}$$ diagonalizing $Ric$ such that $Ric (e_k)=\lambda_ke_k$ for $k=1,2,\ldots,n.$ Using once more that  $T\equiv0,$  we have for all $a,b\geq2$ 
\begin{eqnarray*}
0&=&T_{a1b}=\frac{n-1}{n-2}(R_{ab}\nabla_{1}f-R_{b1}\nabla_{a}f)-\frac{R}{n-2}(g_{ab}\nabla_{1}f-g_{1b}\nabla_{a}f)\\
 &&+\frac{1}{n-2}(g_{ab}R_{1s}\nabla^{s}f-g_{1b}R_{as}\nabla^{s}f)\\
 &=&\frac{n-1}{n-2}R_{ab}|\nabla f|-\frac{R}{n-2}g_{ab}|\nabla f|+\frac{1}{n-2}g_{ab}\lambda_1|\nabla f|.
\end{eqnarray*}
From here it follows that $R_{ab}=\frac{R-\lambda_1}{n-1}g_{ab}$ and then $\lambda_2=\ldots=\lambda_n=\frac{R-\lambda_1}{n-1},$ which gives the fifth assertion. Finally, we use (\ref{codazzi}) as well as (\ref{item1pT}) and (\ref{item4pT}) of the proposition to obtain the last one. So, we complete the proof of the proposition.
\end{proof}

Proceeding with such a metric with $T\equiv0$ we obtain the following lemma.

\begin{lemma}\label{lemcvanish}
Let $(M^n,\,g,\,f)$ be a Miao-Tam critical metric with $T\equiv0.$ Then $C\equiv 0,$ namely, $(M^n,\,g)$ has harmonic Weyl tensor.
\end{lemma}
\begin{proof}
The first part of the proof is standard and it follows the proof of Lemma 4.2 of \cite{CaoChen}. Here we present its proof for the sake of completeness.

First, since $T\equiv0$ we invoke Lemma \ref{lemTCW} to deduce $fC_{ijk}=W_{ijkl}\nabla^{l}f,$ which implies
\begin{equation}
\label{lemcvanisheq1}fC_{ijk}\nabla^{k}f=0.
\end{equation}
We now consider a regular point $p\in M^n,$  with associated level set  $\Sigma.$ We choose any local coordinates $(\theta^{2},\ldots,\theta^{n})$ on $\Sigma$ and split the metric in the local coordinates $(f,\theta^{2},\ldots,\theta^{n})$ as $$g=\frac{1}{|\nabla f|^{2}}df^2+g_{ab}(f,\,\theta) d\theta^{a}d\theta^{b}.$$  Denoting $\partial_{f}=\partial_{1}=\frac{\nabla f}{|\nabla f|^{2}}$ we get
\begin{eqnarray*}
\nabla_{1}f=1 \,\,\,{\rm and}\,\,\,\nabla_{a}f=0,\,\,\,{\rm for}\,\,\,a\geq 2.
\end{eqnarray*}
From (\ref{lemcvanisheq1}) we have $fC_{ij1}=0$ for all $i,j=1,\ldots,n.$ Moreover, for $a,b,c\geq 2,$ by using the Codazzi equation jointly with first and fourth items of Proposition \ref{propvanishT} we have
\begin{equation*}
R_{1abc}=\nabla_{b}^{\Sigma}h_{ac}-\nabla_{c}^{\Sigma}h_{ab}=0.
\end{equation*}

In particular, using $R_{1a}=0$ we get $$W_{1abc}=R_{1abc}=0.$$ Whence, we obtain for $a,b,c\geq 2$ $$fC_{abc}=W_{abcs}\nabla^{s}f=W_{abc1}\nabla^{1}f=0.$$ 

We now claim that $fC_{1ab}=0$ for all $a,b\geq 2.$ To do this, first, we notice that
\begin{equation*}
fC_{1ab}=W_{1abs}\nabla^{s}f=W_{1abi}g^{is}\nabla_{s} f,=W_{1ab1}|\nabla f|^2=-\frac{1}{|\nabla f|^{2}}W(\nabla f,\partial_a,\nabla f,\partial_b).
\end{equation*}
On the other hand, from (\ref{weyl}) we have
\begin{eqnarray}
\label{lemcvanisheq2}
\frac{1}{|\nabla f|^{2}}W(\nabla f,\partial_a,\nabla f,\partial_b)&=&\frac{1}{|\nabla f|^{2}}R(\nabla f,\partial_a,\nabla f,\partial_b)+\frac{R}{(n-1)(n-2)}g_{ab}\nonumber\\&&-\frac{1}{(n-2)}\left(\frac{1}{|\nabla f|^{2}}Ric(\nabla f,\nabla f)g_{ab}+R_{ab}\right).
\end{eqnarray}

We now analyze the second fundamental form in the local coordinate $(f,\theta^{2},\ldots,\theta^{n}).$ It is easy to see that
\begin{equation*}
h_{ab}=\frac{1}{|\nabla f|}\langle \nabla f,\nabla_{a}\partial_{b}\rangle=\frac{1}{|\nabla f|}\langle \nabla f,\Gamma_{ab}^{1}\partial_{f}\rangle=\frac{\Gamma_{ab}^{1}}{|\nabla f|}.
\end{equation*}

Moreover, a standard computation allows us to obtain
\begin{eqnarray*}
\Gamma_{ab}^{1}&=&\frac{1}{2}g^{1j}\big(\partial_{a}g_{bj}+\partial_{b}g_{ja}-\partial_{j}g_{ab}\big)\\&=&\frac{1}{2}g^{11}\big(\partial_{a}g_{b1}+\partial_{b}g_{1a}-\partial_{f}g_{ab}\big)\\
 &=&-\frac{1}{2}\nabla f (g_{ab}).
\end{eqnarray*}

From here it follows that
\begin{equation}
\label{lemcvanisheq3} h_{ab}=-\frac{\nabla f}{2|\nabla f|}(g_{ab}).
\end{equation}

By Proposition \ref{propvanishT}, $|\nabla f|$ is constant on $\Sigma,$ which implies that
\begin{equation}
\label{lemcvanisheq4}[\partial_a,\nabla f]=0.
\end{equation}
Since $\big\langle \frac{\nabla f}{|\nabla f|},\partial_{a}\big\rangle=0,$ we conclude $\nabla_{\frac{\nabla f}{|\nabla f|}}\frac{\nabla f}{|\nabla f|}=0.$ Hence, we can use (\ref{lemcvanisheq4})  to arrive at
\begin{eqnarray*}
\frac{1}{|\nabla f|^2}R(\nabla f,\partial_a,\nabla f,\partial_b)&=&\frac{1}{|\nabla f|}\langle\nabla_{\frac{\nabla f}{|\nabla f|}}\nabla_{a}\partial_{b}-\nabla_{a}\nabla_{\frac{\nabla f}{|\nabla f|}}\partial_b,\nabla f\rangle\nonumber\\
 &=&\frac{1}{|\nabla f|^2}\langle\nabla_{\nabla f}\big(\nabla_{a}^{\Sigma}\partial_b+\nabla_{a}^{\perp}\partial_b\big),\nabla f\rangle- \frac{1}{|\nabla f|}\langle\nabla_{a}\nabla_{\frac{\nabla f}{|\nabla f|}}\partial_b,\nabla f\rangle\nonumber\\
 &=&\frac{\nabla f}{|\nabla f|}(h_{ab})+h_{ac}h^{c}_{b}\nonumber.
\end{eqnarray*}
From this, we deduce
\begin{equation}
\label{9990}
\frac{1}{|\nabla f|^2}R(\nabla f,\partial_a,\nabla f,\partial_b)=\frac{\nabla f}{(n-1)|\nabla f|}H g_{ab}-\frac{H^2}{(n-1)^2}g_{ab}.
\end{equation}
In particular, taking the trace in (\ref{9990}) with respect to $a$ and $b$ we have $$\frac{1}{|\nabla f|^2}Ric(\nabla f,\nabla f)=\frac{\nabla f}{|\nabla f|}H-\frac{H^2}{(n-1)}$$ and then (\ref{9990}) can be written as 
\begin{equation}\label{sol}
R(\nabla f,\partial_a,\nabla f,\partial_b)=\frac{Ric(\nabla f,\nabla f)}{(n-1)}g_{ab}.
\end{equation} By using Proposition \ref{propvanishT} we have $$\frac{1}{|\nabla f|^2}Ric(\nabla f,\nabla f)=\lambda$$ and $$Ric(\partial_a,\partial_b)=\mu g_{ab},$$ for $a,b\geq2$. 

Therefore, substituting (\ref{sol}) in (\ref{lemcvanisheq2}) we get
\begin{eqnarray*}
fC_{1ab}&=&-\frac{1}{|\nabla f|^{2}}W(\nabla f,\partial_a,\nabla f,\partial_b)\\
 &=&-\frac{1}{|\nabla f|^{2}}\frac{Ric(\nabla f,\nabla f)}{(n-1)}g_{ab}+\frac{1}{|\nabla f|^2}\frac{Ric(\nabla f,\nabla f)}{(n-2)}g_{ab}+\frac{R_{ab}}{(n-2)}-\frac{R}{(n-1)(n-2)}g_{ab}\\
 &=&-\frac{\lambda}{(n-1)}g_{ab}+\frac{\lambda}{(n-2)}g_{ab}+\frac{\mu}{(n-2)}g_{ab}-\frac{\lambda+(n-1)\mu}{(n-1)(n-2)}g_{ab}\\
 &=&0,
\end{eqnarray*} which completes our claim. 

Finally, we have $fC_{ijk}=0$ at a point $p$ where  $\nabla f(p)\neq0.$ Therefore, we use Lemma \ref{lemTCW}  to conclude that $fC_{ijk}\equiv0$ in $M^n.$ Using this we obtain $C_{ijk}\equiv0$ in $M\backslash{\partial M}$ and then the proof of the lemma follows from the continuity of the Cotton tensor .
\end{proof}

To finish this section, we shall present a fundamental integral formula.
\begin{lemma}\label{propBT}
Let $(M^n,g,f)$ be a Miao-Tam critical metric. Then
\begin{equation*}
\int_{M}f^2B(\nabla f,\nabla f)dM_{g}=-\frac{1}{2(n-1)}\int_{M}f^2|T|^2 dM_{g}.
\end{equation*}
\end{lemma}
\begin{proof}
From (\ref{cottonwyel}) we can write the Bach tensor  as
$$B_{ij}=\frac{1}{n-2}\left(\nabla_{k}C_{kij}+R_{kl}W_{ikjl}\right).$$ Under this notation we get
\begin{eqnarray*}
f^2B_{ij}&=&\frac{1}{n-2}\left(f^2\nabla_k C_{kij}+f^2 R_{kl}W_{ikjl}\right)\\
 &=&\frac{1}{n-2}\big(\nabla_k\left(f^2C_{kij}\right)-2fC_{kij}\nabla_{k}f+f^2R_{kl}W_{ikjl}\big).
\end{eqnarray*}
We now use Lemma \ref{lemTCW} jointly with (\ref{cottonwyel}) to obtain
\begin{eqnarray*}
 f^2B_{ij}&=&\frac{1}{n-2}\left(\nabla_k\left[f(W_{kijl}\nabla_{l}f+T_{kij})\right]-2fC_{kij}\nabla_{k}f+f^2R_{kl}W_{ikjl}\right)\\
 &=&\frac{1}{n-2}\big(\nabla_k\left(fT_{kij}\right)+f\nabla_{k}W_{kijl}\nabla_{l}f+fW_{kijl}\nabla_{k}\nabla_{l}f\\
 &&+W_{kijl}\nabla_{k}f\nabla_{l}f-2fC_{kij}\nabla_{k}f+f^2R_{kl}W_{ikjl}\big)\\
 &=&\frac{1}{n-2}\Big(\nabla_{k}\left(fT_{kij}\right)+\frac{n-3}{n-2}fC_{jki}\nabla_{k}f+f\left(\nabla_{k}\nabla_{l}f-fR_{kl}\right)W_{kijl}\\
 &&+W_{kijl}\nabla_{k}f\nabla_{l}f+2f C_{ikj}\nabla_{k}f\Big).
\end{eqnarray*}
We recall that the Weyl tensor is trace-free on any pair of indices. Next, by using (\ref{fundeqtens}) we deduce
\begin{equation*}
f^2B_{ij}=\frac{1}{n-2}\Big[\nabla_{k}(f T_{kij})+\frac{n-3}{n-2}fC_{jki}\nabla_{k}f+2fC_{ikj}\nabla_{k}f+W_{kijl}\nabla_{k}f\nabla_{l}f\Big].
\end{equation*}
Whence, we obtain
\begin{equation}
\label{123456}
f^2B(\nabla f,\nabla f)=\frac{1}{n-2}\nabla_{k}\left(fT_{kij}\right)\nabla_{i}f\nabla_{j}f.
\end{equation} On the other hand, we notice that $$\nabla_{k}\left(fT_{kij}\right)\nabla_{i}f\nabla_{j}f=\nabla_{k}\left(fT_{kij}\nabla_{i}f\nabla_{j}f\right)- fT_{kij}\nabla_{k}\nabla_{i}f\nabla_{j}f-fT_{kij}\nabla_{i}f\nabla_{k}\nabla_{j}f.$$ Now, on integrating (\ref{123456}) over $M$ and  using Stokes  formula we arrive at

\begin{eqnarray*}
\int_M f^2 B(\nabla f,\nabla f)dM_{g}
 &=&-\frac{1}{n-2}\left(\int_M fT_{kij}\nabla_{k}\nabla_{i}f\nabla_{j}f dM_{g}+\int_M fT_{kij}\nabla_{i}f\nabla_{k}\nabla_{j}fdM_{g}\right)\\
 &=&-\frac{1}{n-2}\left(\int_M f^2T_{kij}R_{ki}\nabla_{j}f dM_{g}+\int_M f^2T_{kij}R_{kj}\nabla_{i}fdM_{g}\right),
\end{eqnarray*}
in the last equality we have used once more (\ref{fundeqtens}). Finally, we change $k$ by $i$ above and then using the properties of $T$ defined in (\ref{T})  we achieve
\begin{eqnarray*}
\int_M f^2 B(\nabla f,\nabla f)dM_{g}&=&-\frac{1}{2(n-2)}\left(\int_M f^2T_{kij}R_{ki}\nabla_{j}f dM_{g}+\int_M f^2T_{kij}R_{kj}\nabla_{i}fdM_{g}\right)\\
 &&-\frac{1}{2(n-2)}\left(\int_M f^2T_{ikj}R_{ik}\nabla_{j}f dM_{g}+\int_M f^2T_{ikj}R_{ij}\nabla_{k}f dM_{g}\right)\\
 &=&\frac{1}{2(n-2)}\int_M f^2T_{ikj}\left(R_{kj}\nabla_{i}f-R_{ij}\nabla_{k}f\right)dM_{g}\\
 &=&-\frac{1}{2(n-1)}\int_M f^2|T|^2 dM_{g},
\end{eqnarray*}
that was to be proved.
\end{proof}

\section{Proof of the Results}

\subsection{Proof of Theorem \ref{thm1}}
\begin{proof}
First, since $M^4$ satisfies $\int_M f^2B(\nabla f,\nabla f)dM\geq 0$ we invoke Lemma \ref{propBT} to conclude that $T\equiv0.$  Therefore,  from Lemma \ref{lemcvanish} we have $C\equiv0.$ Hence, we use Lemma \ref{lemTCW} to obtain
\begin{equation*}
W_{ijkl}\nabla^{l}f=0.
\end{equation*}
We now consider a point  $p\in M^4$ such that $\nabla f (p)\neq0.$ Choosing an orthonormal frame $\{ e_1, e_2, e_3, e_4\}$ with $e_1=\frac{\nabla f}{|\nabla f|}$ at the point $p$ we arrive at
\begin{equation}
\label{123er}
W_{ijk1}=0,
\end{equation} for all $1\le i,\,j,\,k\le 4.$

We now claim that $W_{ijkl}=0$ whenever $\nabla f (p)\neq 0.$ Indeed, recalling that the Weyl tensor is trace-free on any pair of indices, we have
\begin{equation*}
W_{2121}+W_{2222}+W_{2323}+W_{2424}=0.
\end{equation*} By using (\ref{123er}) we have $$W_{2323}=-W_{2424}.$$ In a similar way we have $$W_{2424}=-W_{3434}=W_{2323}.$$ From here it follows that $W_{2323}=0.$ Moreover, we also have
\begin{equation*}
W_{1314}+W_{2324}+W_{3334}+W_{4344}=0.
\end{equation*} Therefore, $W_{2324}=0.$ This proves that $W_{abcd}=0$ unless $a,b,c,d$ are all distinct. But, there are only three choices for the indices. This concludes the proof of our claim.

Hereafter, choosing appropriate coordinates (e.g. harmonic coordinates) we conclude that $f$ and $g$ are analytic, see for instance Theorem 2.8 in \cite{Corvino} (or Proposition 2.1 in \cite{CEM}). Whence, we conclude that $f$ can not vanish identically in a non-empty open set. So, the set  of regular points is dense in $M^4.$ This allows us to conclude that  $M^4$ is locally conformally flat and we are in position to use Theorem \ref{theoremMT} (see also Theorem 1.2 of  \cite{miaotamTAMS}) to conclude that $(M^4 ,\,g)$ is isometric to a geodesic ball in a simply connected space form $\Bbb{R}^{4},$ $\Bbb{H}^{4}$ or $\Bbb{S}^{4}.$ This is what we wanted to prove.

\end{proof}

\subsection{Proof of Theorem \ref{thm2}}

\begin{proof}
The first part of the proof will follow  \cite{CCCMM}. To begin with, we consider $(M^3,\,g,\,f)$ be a simply connected, compact Miao-Tam critical metric with boundary isometric to a standard sphere $\Bbb{S}^2.$ Next, we recall that the Cotton tensor can be written as $C_{ijk}=\nabla_iA_{jk}-\nabla_{j}A_{ik}.$ From here it follows that
\begin{eqnarray*}
\nabla_iB_{ij}&=&\nabla_i\nabla_k C_{kij}\\
&=& \big(\nabla_i\nabla_k-\nabla_k\nabla_i\big)\nabla_k A_{ij}.
\end{eqnarray*}
Whence, the previous commutator term implies $$\nabla_iB_{ij}=-R_{il}\nabla_{l}A_{ij}+R_{kl}\nabla_{k}A_{lj}+R_{ikjl}\nabla_kA_{il}$$ and then we get
\begin{eqnarray}
\label{n1}
\nabla_iB_{ij}=R_{ikjl}\nabla_kA_{il}.
\end{eqnarray}
We now remember that  $W\equiv 0$ in dimension three. So, a straightforward computation involving (\ref{WS}) and (\ref{n1}) gives
\begin{eqnarray}
\label{n2}
\nabla_iB_{ij} \nabla_{j} f &=&A_{ik}C_{kji}\nabla_{j} f+A_{ik}\nabla_jA_{ki}\nabla_{j} f+A_{jl}\nabla_iA_{il}\nabla_{j} f\nonumber\\
 &&-A_{il}\nabla_jA_{il}\nabla_{j} f-A_{jk}g_{il}C_{kil}\nabla_{j} f-A_{jk}g_{il}\nabla_iA_{kl}\nabla_{j} f\nonumber\\
 &=&-R_{ik}C_{jki}\nabla_{j} f\nonumber\\
 &=&-\frac{1}{2}\left(C_{jki}R_{ik}\nabla_jf+C_{kji}R_{ij}\nabla_kf\right),
\end{eqnarray} where we have used that $C$ is skew-symmetric in the first two indices and that $C$ is trace-free in any two indices. Now, we combine (\ref{n2}) with (\ref{T}) to deduce
\begin{eqnarray*}
(div B)(\nabla f)&=&\frac{1}{2}C_{kji}\left(R_{ik}\nabla_jf-R_{ij}\nabla_kf\right)\\
&=&\frac{1}{4}C_{kji}T_{kji}.
\end{eqnarray*} Using once more that $W\equiv 0$ jointly with Lemma \ref{lemTCW} we arrive at $$(div B)(\nabla f)=\frac{f}{4}|C|^2.$$ Therefore, our assumption together with the continuity of the Cotton tensor implies $C\equiv 0$ in $M^3$ and then $(M^3,\,g)$ is locally conformally flat. Finally, it suffices to use Theorem \ref{theoremMT} to get the promised result.
\end{proof}

\begin{acknowledgement}
The authors want to thank the referees for their careful reading and helpful suggestions. 
The third author would like to thank Huai-Dong Cao for valuable conversations about Bach-flat metrics.  He would like to extend his special thanks to Pengzi Miao for helpful conversations about critical metrics. Moreover, he wish to express his gratitude for the excellent support during his stay at ICTP-Italy, where part of this work was started. Finally, he wishes to thank the Department of Mathematics - Lehigh University  for the warm hospitality and for the fruitful research environment.\end{acknowledgement}

\end{document}